\documentclass[12pt,twoside,english,reqno,a4paper,final]{amsart}
\usepackage[utf8]{inputenc}
\usepackage[english]{babel}
\usepackage{cancel}
\usepackage[lined,boxed]{algorithm2e}
\usepackage{float} 
\usepackage{hyperref}
\usepackage[notref,notcite]{showkeys}
\usepackage{fullpage}
\usepackage{enumitem}

\usepackage{listings,graphicx,amsmath,varioref,amscd,amssymb,color,xcolor,bm,amsthm,amsfonts,graphics,lineno,float,comment}
\usepackage{mathrsfs} 

\newcommand{\de}{\partial}
\newcommand{\z}{{\bf z}}
\newcommand{\R}{\mathbb R}
\newcommand{\Div}{\hbox{div\,}}

\theoremstyle{plain}
\newtheorem{theorem}{Theorem}[section]
\newtheorem{proposition}[theorem]{Proposition}

\newtheorem{lemma}[theorem]{Lemma}

\theoremstyle{definition}
\newtheorem{defin}[theorem]{Definition}

\newtheorem{remark}[theorem]{Remark}

\theoremstyle{remark}

\numberwithin{equation}{section}
\newcommand{\sg}{\hbox{\, sign\,}}

\author[F. Della Pietra]{Francesco Della Pietra}
\author[F. Oliva]{Francescantonio Oliva}
\author[S. Segura]{Sergio Segura de Le\'on}

\address{Francesco Della Pietra - Dipartimento di Matematica e Applicazioni, Universit\`a di Napoli Federico II, Via Cintia, Monte S. Angelo, 80126 Napoli, Italy}
\email{f.dellapietra@unina.it}

\address{Francescantonio Oliva - Dipartimento di Matematica e Applicazioni, Universit\`a di Napoli Federico II, Via Cintia, Monte S. Angelo, 80126 Napoli, Italy}
\email{francescantonio.oliva@unina.it}
\address{Sergio Segura de Le\'on - Departament d'An\`alisi Matem\`atica, Universitat de Val\`encia, Dr. Moliner 50,
46100 Burjassot, Val\`encia, Spain}
\email{sergio.segura@uv.es}


\begin{document}

\title{Behaviour of solutions to $p$-Laplacian with Robin boundary conditions as $p$ goes to $1$}

\maketitle

\begin{abstract}
We study the asymptotic behaviour, as $p\to 1^{+}$, of the solutions of the following inhomogeneous Robin boundary value problem:
\begin{equation}
\label{pbabstract}
\tag{P}
	\begin{cases}
		\displaystyle -\Delta_p u_p = f & \text{ in }\Omega,\\
		\displaystyle |\nabla u_p|^{p-2}\nabla u_p\cdot \nu +\lambda |u_p|^{p-2}u_p = g& \text{ on } \partial\Omega,
	\end{cases}
\end{equation}
where $\Omega$ is a bounded domain in $\mathbb R^{N}$ with sufficiently smooth boundary, 
$\nu$ is its unit outward normal vector and $\Delta_p v$ is the $p$-Laplacian operator with $p>1$. The data $f\in L^{N,\infty}(\Omega)$ (which denotes the Marcinkiewicz space) and $\lambda,g$ are bounded functions defined on $\partial\Omega$ with $\lambda\ge0$. 
We find the threshold below which the family of $p$--solutions goes to 0 and above which this family blows up. As a second interest we deal with the $1$-Laplacian problem  formally arising by taking $p\to 1^+$ in \eqref{pbabstract}.
\end{abstract}

\begin{center}
\begin{minipage}{.7\textwidth}
\tableofcontents
\end{minipage}
\end{center}


\section{Introduction}
The aim of this paper is twofold. We first deal with the asymptotic behaviour of solutions to inhomogeneous Robin boundary value problems with $p$-Laplacian as principal operator and then we analyse existence of solution for the limit problem as $p\to 1^+$. To be more precise, let $\Omega$ be an open bounded subset of $\mathbb{R}^N$ ($N\ge 2$) with smooth boundary and let $\nu$ denote its unit outward normal vector. We consider problems
\begin{equation}
	\label{pbintro}
	\begin{cases}
		\displaystyle -\Delta_p u_p = f & \text{ in }\Omega,\\
		\displaystyle |\nabla u_p|^{p-2}\nabla u_p\cdot \nu +\lambda |u_p|^{p-2}u_p = g& \text{ on } \partial\Omega,
	\end{cases}
\end{equation}
where  $\Delta_p v=\textrm{div}(|\nabla v|^{p-2}\nabla v)$ is the $p$-Laplacian operator with $p>1$, $f$ belongs to the Marcinkiewicz space $L^{N,\infty}(\Omega)$ and $\lambda,g$ are bounded functions defined on $\partial\Omega$ with $\lambda\ge0$ not identically null. In this paper, we will study the behaviour of solutions $u_p$ as $p\to 1^+$ and, when this family converges to an almost everywhere finite function $u$, we will check that $u$ is a solution to the limit problem.
\medskip

Let us observe that problem \eqref{pbintro} formally turns into a Dirichlet problem once that $\lambda=\infty$, or into a Neumann problem if $\lambda\equiv 0$. In these extremal cases, the study of the asymptotic behaviour with respect to $p\to 1^+$ in problems driven by the $p$--Laplacian is nowadays classical and widely studied. 

\subsection{Asymptotic behaviour}
Without the purpose of being exhaustive, we present some of the results which mostly motivated our work.
\\
The Dirichlet case presents a huge literature; in \cite{ct,K} the authors observe that solutions to \eqref{pbintro}, obtained as a minimum of a suitable functional, converge to a minimum of the functional written for $p=1$. Since $W^{1,1}(\Omega)$ is not reflexive, the limit is only expected to belong to $BV(\Omega)$. It is shown in \cite{K} that, when $f\equiv 1$, the family $u_p$ goes to 0 or to $\infty$, depending on the domain. 
This degeneration/blow up phenomenon was extended in \cite{ct}. It is shown that if $\|f\|_{N,\infty}<\tilde{\mathcal{S}}_1$ then $u_p\to 0$ almost everywhere in $\Omega$ as $p\to 1^+$ where $\tilde{\mathcal{S}}_1$ is the best constant in the Sobolev embedding from $W^{1,1}(\Omega)$ into the Lorentz space $L^{\frac{N}{N-1},1}(\Omega)$ (see \cite{alvino}). In the critical case $\|f\|_{N,\infty} = \tilde{\mathcal{S}}_1$ the solutions $u_p$ converge almost everywhere to a function $u$ as $p\to 1^+$ which is, in general, not null. Finally, if $\|f\|_{N,\infty}> \tilde{\mathcal{S}}_1$, examples of $u_p$ blowing up as $p\to 1^+$ on a subset of $\Omega$ of positive measure are made explicit. 
\\
This result has been specified in \cite{mst2} in the following sense: if $\|f\|_{W^{-1,\infty}(\Omega)}<1$ then $u_p$ degenerates to zero, if $\|f\|_{W^{-1,\infty}(\Omega)}=1$ then $u_p$ converges to an almost everywhere finite function and, finally, if $\|f\|_{W^{-1,\infty}(\Omega)}>1$ then $u_p$ blows up as $p\to 1^+$.
\\
For the Neumann case, we mention \cite{mrst}; here, in case $f\equiv 0$ and under the compatibility condition given by $\int_{\partial\Omega} g\, d\mathcal H^{N-1} = 0$, the authors show once again the degeneration/blow up phenomenon. If a suitable norm of $g$ is small enough, then $u_p$ converges almost everywhere in $\Omega$ to a function which is almost everywhere finite. By the way, if the same norm is large enough, $u_p$ converges to a function which is infinite on a set of positive measure.

\medskip

Therefore, it should be expected that the solutions $u_p$ to \eqref{pbintro} experience the same phenomena described above. Then a natural question is  determining the threshold which describes it.
As we will see, a key role is played by the following quantity 
\begin{equation*}
	M(f,g,\lambda)=\sup_{u\in W^{1,1}(\Omega)\backslash\{0\}}\displaystyle\frac{\displaystyle\int_\Omega f u\, dx+ \displaystyle\int_{\partial\Omega}g u\, d\mathcal H^{N-1}}{\displaystyle\int_\Omega |\nabla u|\, dx + \int_{\partial\Omega}\lambda |u|\, d\mathcal H^{N-1}},
\end{equation*}
which is finite once that $f\in L^{N,\infty}(\Omega)$ and $g\in L^\infty(\partial\Omega)$. We point out that the de\-no\-mi\-na\-tor defines a norm in $W^{1,1}(\Omega)$ which is equivalent to the usual one (see \cite[Section 2.7]{N}).
\\
Using $M(f,g,\lambda)$, our first result can be described as follows: if $M\le 1$ then the sequence $u_p$ is bounded in $BV(\Omega)$ with respect to $p$ and it converges to zero if $M< 1$. Moreover, the result is optimal in the sense that if $M>1$, then $u_p$ blows up on a set of positive measure as $p\to 1^+$ (see Theorem \ref{mainteo} below). Let us also mention that explicit examples show that when $M=1$ the limit function is not null in general (see Section \ref{radialSec} below).
This means that the asymptotic behaviour of $u_p$ is completely settled from $M$.
\\
A further remark on this threshold $M$ is in order. We stress that $M$ depends on both the volumetric datum $f$ and the boundary datum $g$. As far as we know, it is the first time that the phenomenon of degeneracy/blow up is studied when two data occur. For a single datum an essential tool is the H\"older inequality. In our setting this inequality does not lead to the desired value. So, we needed to extend it in order to handle both data (for details we refer to the appendix).

\subsection{Limit problem}
After studying the asymptotic behaviour, we mean to study the $1$-Laplace limit problem. That is we deal with existence of a solution, intended suitably (see Definition \ref{def} below), to the following problem
\begin{equation}
	\label{pbintrolimit}
	\begin{cases}
		\displaystyle -\Delta_1 u = f & \text{ in }\Omega,\\
		\displaystyle \frac{D u}{|D u|}\cdot \nu +\lambda \sg{u} = g& \text{ on } \partial\Omega,
	\end{cases}
\end{equation}
which is formally the limit as $p\to 1^+$ of \eqref{pbintro}. Here $\Delta_1 u:= \operatorname{div}\left(\frac{Du}{|Du|}\right)$ is the $1$-Laplacian operator.
\\
It is worth highlighting that, among others, the $1$-Laplace equations are strongly related to image processing, torsion and mean curvature problems (see \cite{abcm0, abcm, BCRS,ka,moser,OsSe,Sapiro}). From the mathematical point of view, there is huge literature concerning existence, uniqueness and regularity of solutions to problems involving the $1$-Laplace operator under Dirichlet boundary conditions; even the case $\lambda=0$ has been dealt with but, unsurprisingly, the literature is more limited. The study of this type of problems is a very active branch as shown by recent works such as \cite{aop,dgop,D,liliu, ms,mp,ss}.
\\
\medskip

Nevertheless, in all the papers cited above, a common denominator is that the solutions belong, in general, only to the $BV$-space. This clearly plays a role in the way the singular quotient $|Du|^{-1}Du$ needs to be intended both in $\Omega$ and on $\partial\Omega$. In \cite{D} and \cite{abcm} this difficulty is overcome for the first time by using a bounded vector field $\z$ whose divergence is a function enjoying some regularity.
Just have in mind that this allows  to define a distribution $(\z, Du)$ which couples one of these bounded vector fields and the gradient of a $BV$--function (see \cite{An} and \cite{CF}, in Section \ref{anzellottisec} below is briefly recalled). In other words this pairing, which is nothing more than the scalar product if the involving terms are regular enough, is a way to give sense to the singular quotient through a bounded vector field $\z$ satisfying $\|\z\|_\infty\le 1$ and $(\z,Du)=|Du|$, while the equation holds as $-\operatorname{div}\z = f$. 
\\
For a vector field $\z$ of this type it is also possible to define a weak normal trace (denoted by $[\z,\nu]$ below) which enters strongly in the definition of the boundary condition. Indeed, another common feature for $1$-Laplace equations is that the boundary datum is not necessarily attained in the sense of traces. Just to give an idea, in the Dirichlet problem, a standard request is  $[\z,\nu]\in \sg(-u)$ on  $\partial\Omega$. On the contrary, the Neumann boundary condition holds pointwise as shown in \cite{mrst}. In our framework, situated in between, we cannot expect the boundary condition to hold. Nevertheless, it should be satisfied when $\lambda$ tends to 0.

\medskip

As far as we know, the only related paper involving the 1-Laplace operator and a  boundary condition of Robin's type is \cite{mrs}. The authors deal with $f \equiv 0$ jointly  with a boundary condition as 
$$\displaystyle \frac{D u}{|D u|}\cdot \nu +\lambda u = g$$ 
where $\lambda$ is a positive constant and  $g\in L^2(\partial\Omega)$.  Note, however, that this condition is slightly different from ours. More general data $g$ can be handled in \cite{mrs} owing to the presence of the absorption term $\lambda u$. It also provides a regularizing effect on the solution which is proved to always lie in $L^2(\partial\Omega)$; this is something that in general we will not expect for solutions to \eqref{pbintrolimit}. 
\\
\medskip

Thus, we deal with existence of a solution to \eqref{pbintrolimit} under the assumptions $f\in L^{N,\infty}(\Omega)$, $g\in L^\infty(\partial\Omega)$ and $0\le \lambda\in L^\infty(\partial\Omega)$ (see Theorem \ref{existence}). Working by approximation through problems \eqref{pbintro}, the result is achieved by requiring that $M\le 1$. It worth mentioning that the presence of $\lambda \in L^\infty(\partial\Omega)$ (see also Section \ref{secL1} for the extension to the merely integrable case) produces extra difficulties with respect to the Dirichlet and Neumann cases. Indeed, for the equation in $\Omega$ a lower semicontinuity argument is needed (see Lemma \ref{low-sem1} below) which has also its own interest besides problem \eqref{pbintrolimit}. Even the boundary condition presents some challenges. Indeed, in order to characterize the solution on the boundary we will use an auxiliary function $\beta$ which is actually the sign function under some restriction on the data and in the zone where $\lambda$ is positive. If $|g-\lambda\sg(u)|\le 1$, the boundary condition holds pointwise on the set $\{\lambda>0\}\cap \{u\big|_{\partial\Omega}\ne 0\}$. Otherwise, if $|g-\lambda\sg(u)|>1$, then the boundary condition should be interpreted as $||Du|^{-1}Du\cdot \nu|$ is forced to be as high as possible. This is basically the weak way we mean the boundary term (see also Remark \ref{remdef} below). This feature is similar to that obtained in \cite[Definition 2.3 and Remark 2.7]{mrs}, but our approach is different.

\subsection{Plan of this paper}
The next Section is on preliminaries; the theory underlying the pairings $(\z, Du)$ and the weak trace $[\z, \nu]$ is sketching there. Section \ref{sec_behaviour} is dedicated to the asymptotic behaviour of $u_p$ as $p\to 1^+$. In Section \ref{sec:limitproblem} we consider the $1$-Laplace problem which formally arises by taking $p\to 1^+$ into \eqref{pbintro}. In Section \ref{sec:examples} we give some extensions and examples concerning the results of the previous two sections. Finally, in the appendix, we briefly consider two inequalities which are used throughout the paper.

\section{Preliminaries}

\subsection{Notation}
For a given function $v$ we denote by $v^+=\max(v,0)$ and by $v^-= -\min (v,0)$. For a fixed $k>0$, we define the truncation functions $T_{k}:\mathbb{R}\to\mathbb{R}$ as follows
\begin{align*}
	T_k(s):=&\max (-k,\min (s,k)).
\end{align*}
We denote by $|E|$ and by $\mathcal H^{N-1}(E)$ respectively the Lebesgue measure and the $(N-1)$--dimensional Hausdorff measure of a set $E$. Moreover $\chi_{E}$ stands for its characteristic function. 

If no otherwise specified, we denote by $C$ several positive constants whose value may change from line to line and, sometimes, on the same line. These values will only depend on the data but they will never depend on the indexes of the sequences we introduce below.

\subsection{Functional spaces}

Throughout this paper, $\Omega\subset \R^N$ (with $N\ge2$) stands for an open bounded set with, at least, Lipschitz boundary. The unit outward normal vector, which exists $\mathcal H^{N-1}$--a.e. on $\partial\Omega$, is denoted by $\nu$.

 We denote by $L^q(E)$ the usual Lebesgue space of $q$--summable functions on $E$. The symbol $L^q(\partial\Omega, \lambda)$ stands for the Lebesgue space having weight $\lambda$.
 
 A function $f$ belongs to the Marcinkiewicz (or weak Lebesgue) space $L^{N,\infty}(\Omega)$ when $|\{|f|>t\}|\le C t^{-N}$, for any $t>0$. We recall that $L^{N}(\Omega)\subset L^{N,\infty}(\Omega)\subset L^{N-\varepsilon}(\Omega)$, for any $\varepsilon>0$. We refer to \cite{hunt} for an overview on these spaces. 
 \\
 \medskip
We will denote by $W^{1,p}(\Omega)$ the usual Sobolev space, of measurable functions having weak derivative in $L^{p}(\Omega)^N$. It is a Banach space when endowed with the usual norm.
It is well-known that functions in Sobolev spaces have a trace on the boundary, this fact allows us to write $u\big|_{\partial\Omega}$. Moreover, if $u\in W^{1,1}(\Omega)$, then $u\big|_{\partial\Omega}\in L^1(\partial\Omega)$ and the embedding $W^{1,1}(\Omega)\to L^1(\partial\Omega)$ is onto.
On the other hand, the Sobolev space $W^{1,1}(\Omega)$ is compactly embedded in $L^1(\Omega)$ and continuously embedded into the Lorentz space $L^{\frac{N}{N-1},1}(\Omega)$ (see\cite{alvino}). Since this Lorentz space has $L^{N,\infty}(\Omega)$ as its dual (see \cite{hunt}), it follows that $fu\in L^1(\Omega)$ for every $f\in L^{N,\infty}(\Omega)$ and every $u\in W^{1,1}(\Omega)$.
Finally, for a nonnegative $\lambda\in L^\infty(\partial\Omega)$ not identically null, the norm defined in $W^{1,1}(\Omega)$ as
\begin{equation}\label{norma}
\|v\|_\lambda=\int_\Omega |\nabla v|\, dx+\int_{\partial\Omega}\lambda(x)|v|\, d\mathcal H^{N-1}\end{equation}
is equivalent to the usual norm in $W^{1,1}(\Omega)$ (see \cite[Section 2.7]{N}). 
\\
\medskip
The space of  functions of bounded variation is defined as
$$BV(\Omega):=\{ u\in L^1(\Omega)\> :\> Du \> \hbox{ is a Radon measure with finite variation} \},$$
which is a Banach space. 

Most of the features of $W^{1,1}(\Omega)$ also hold for $BV(\Omega)$, since the proofs can easily be adapted by approximation. In this paper, we will use the following facts:
\begin{enumerate}
\item equation \eqref{norma} defines a norm in $BV(\Omega)$ equivalent to the usual one;
\item the trace operator $BV(\Omega)\to L^1(\partial\Omega)$ is continuous and onto;
\item the embedding $BV(\Omega)\to L^1(\Omega)$ is compact;
\item  the embedding $BV(\Omega)\to L^{\frac{N}{N-1},1}(\Omega)$ is continuous.
\end{enumerate}
As a consequence of the last property, $fu\in L^1(\Omega)$ for every $f\in L^{N,\infty}(\Omega)$ and every $u\in BV(\Omega)$.
We refer to \cite{afp} for a complete account on this space.

\subsection{$L^\infty$-divergence vector fields}
\label{anzellottisec}
We briefly present the $L^\infty$-divergence-measure vector fields theory (see \cite{An} and \cite{CF}). We denote
\[X(\Omega):=\{ \z\in L^\infty(\Omega, \R^N) : \operatorname{div}\z \in L^{N,\infty}(\Omega)\}.\]


In \cite{An} the distribution $(\z,Dv): C^1_c(\Omega)\to \mathbb{R}$ is defined as 
\begin{equation*}\label{dist1}
	\langle(\z,Dv),\varphi\rangle:=-\int_\Omega v\varphi\operatorname{div}\z-\int_\Omega
	v\z\cdot\nabla\varphi,\quad \varphi\in C_c^1(\Omega),
\end{equation*}
which is well defined if $v\in BV(\Omega)$ and $\z$ is a bounded vector field such that its divergence belongs to $ L^N(\Omega)$. Moreover $(\z, Dv)$ is a Radon measure satisfying
\begin{equation*}\label{finitetotal1}
	\left| \int_B (\z, Dv) \right|  \le  \int_B \left|(\z, Dv)\right| \le  ||\z||_{L^\infty(U,\R^N)} \int_{B} |Dv|\,,
\end{equation*}
for all Borel sets $B$ and for all open sets $U$ such that $B\subset U \subset \Omega$.
\\
Let us also remark that, in \cite{An}, it is shown the existence of a weak trace on $\partial \Omega$ for the normal component of a bounded vector field $\z$ such that $\operatorname{div}\z \in L^1(\Omega)$. This is denoted by
$[\z, \nu]$ where $\nu(x)$ is the outward normal unit vector. Then it is proven that
\begin{equation*}\label{des1}
	||[\z,\nu]||_{L^\infty(\partial\Omega)}\le ||\z||_{\infty}\,.
\end{equation*}
Finally a Green formula holds:
$$	\int_{\Omega} v \operatorname{div}\z + \int_{\Omega} (\z, Dv) = \int_{\partial \Omega} v[\z, \nu] \ d\mathcal H^{N-1},$$
where  $\z \in L^\infty(\Omega,\R^N)$, $\operatorname{div}\z \in L^N(\Omega)$  and $v\in BV(\Omega)$. Let us stress that all previous results can be easily extended to the case where $\z\in X(\Omega)$ and $u\in BV(\Omega)$ thanks to the continuous embedding of $BV(\Omega)$ into $L^{\frac{N}{N-1},1}(\Omega)$.

\section{Asymptotic behaviour as $p\to 1^+$}
\label{sec_behaviour}
Let $\Omega$ be a bounded open set of $\mathbb{R}^N$ ($N\ge 2$) with Lipschitz boundary. We are interested into taking $p\to 1^+$ in the following Robin problem:
\begin{equation}
\label{pbmain}
\begin{cases}
\displaystyle -\Delta_p u_p = f & \text{ in }\Omega,\\
\displaystyle |\nabla u_p|^{p-2}\nabla u_p\cdot \nu +\lambda |u_p|^{p-2}u_p = g& \text{ on } \partial\Omega,
\end{cases}
\end{equation}
where $f\in L^{N,\infty}(\Omega)$, $\lambda \in L^{\infty}(\partial\Omega)$ is nonnegative but not identically null and finally $g\in L^\infty(\partial\Omega)$. The existence of $u_p\in W^{1,p}(\Omega)$ satisfying \eqref{pbmain} follows from \cite{ll}. We remark that $u_p$ can also be obtained as a minimum of a suitable functional (see Section \ref{sec:variational} below). For this section we are interested in the asymptotic behaviour of $u_p$ as $p\to 1^+$.

\medskip

To begin with, we introduce the key quantity
\begin{equation*}
	M(f,g,\lambda)=\sup_{u\in W^{1,1}(\Omega)\backslash\{0\}}\displaystyle\frac{\displaystyle\int_\Omega f u\, dx+ \displaystyle\int_{\partial\Omega}g u\, d\mathcal H^{N-1}}{\|u\|_\lambda},
\end{equation*}
which is always finite once that $f\in L^{N,\infty}(\Omega)$ and $g\in L^\infty(\partial\Omega)$.
In particular we show that if $M(f,g,\lambda)\le 1$, then we have an estimate of the family $u_p$ in $BV(\Omega)$; otherwise, as we will see, the solutions $u_p$ blow up on a set of positive measure as $p$ approaches 1. This is the content of main theorem of this section:
\begin{theorem}\label{mainteo}
	Given $f\in L^{N,\infty}(\Omega)$,  $\lambda \in L^{\infty}(\partial\Omega)$  nonnegative but not identically null and  $g\in L^\infty(\partial\Omega)$, let $u_p$ be a solution to \eqref{pbmain}. Then, up to subsequences, it holds:
	\begin{enumerate}
		\item [i)]  if $M(f,g,\lambda)< 1$ then $u_p$ converges almost everywhere in $\Omega$ to zero as $p\to 1^+$;
		\item [ii)] if $M(f,g,\lambda)= 1$ then $u_p$ converges almost everywhere in $\Omega$ to a function $u$ as $p\to 1^+$ which is almost everywhere finite;
		\item [iii)] if $M(f,g,\lambda) > 1$ then $|u_p|$ blows up either on a subset of $\Omega$ of positive Lebesgue measure or on a subset of $\partial\Omega$ of positive $\mathcal H^{N-1}$ measure.
	\end{enumerate}
\end{theorem}

\begin{remark}
	It is worth to  highlighting that in Section \ref{radialSec} below the results of the previous theorem are explicitly computed for the case $\Omega$ as a ball. In particular, let us note that in case $M=1$ one can actually find explicit  examples of limit functions $u$ which are not null.
\end{remark}

\begin{remark}
In the homogeneous Dirichlet case, that is when formally $\lambda=+\infty$, then
\[
M=\sup_{u\in W_{0}^{1,1}(\Omega)\backslash\{0\}}\displaystyle\frac{\displaystyle\int_\Omega f u\, dx}{\displaystyle\int_{\Omega}|\nabla u|dx}.
\]
By the Hardy-Littlewood and Sobolev inequalities, it is easy to see that
\[
M\le \frac{\|f\|_{L^{N,\infty}(\Omega)}}{N\omega_{N}^{1/N}},
\]
where $\omega_{N}$ is the volume of the unit ball in $\R^{N}$. This implies that the smallness condition on $f$  considered in \cite{ct} in order to obtain a finite limit for $u_{p}$, namely $\|f\|_{L^{N,\infty}(\Omega)}\le N\omega_{N}^{1/N}$, always implies that $M\le 1$ (see also \cite{mst1}).
\end{remark}

\medskip

We start stating and proving the uniform estimate under the smallness condition on $M(f,g,\lambda)$.

\begin{lemma}\label{lemmastimafond}
	Let $f\in L^{N,\infty}(\Omega)$, let $\lambda \in L^{\infty}(\partial\Omega)$ be nonnegative but not identically null and let $g\in L^\infty(\partial\Omega)$. If $u_p$ is a solution to \eqref{pbmain}, then it holds
	\begin{equation*}\label{stimaulambda}
			\| u_p\|_\lambda 		\le M(f,g,\lambda)^{\frac1{p-1}}\left[|\Omega|+\int_{\partial\Omega}\lambda\, d\mathcal H^{N-1}\right].
	\end{equation*}
	 Furthermore if $M(f,g,\lambda)\le 1$ then $u_p$ is bounded in $BV(\Omega)$ with respect to $p$ and it converges, up to a subsequence, *-weakly in $BV(\Omega)$ to a function $u$ as $p\to 1^+$. In particular if $M(f,g,\lambda) < 1$ then $u$ is identically null. 	
\end{lemma}
\begin{proof}
Let us take $u_p$ as test function in \eqref{pbmain}, it yields
\begin{equation*}
	\begin{aligned}
	\int_\Omega|\nabla u_p|^pdx+\int_{\partial\Omega}\lambda|u_p|^p d\mathcal H^{N-1}
	&=\int_\Omega f u_p\, dx+\int_{\partial\Omega}g u_p\, d\mathcal H^{N-1}\\
	&\le M(f,g,\lambda)\left[\int_\Omega|\nabla u_p|\, dx+\int_{\partial\Omega}\lambda|u_p|\,  d\mathcal H^{N-1}\right]
	\end{aligned}
\end{equation*}
Denoting
\[A^p=\int_\Omega|\nabla u_p|^pdx+\int_{\partial\Omega}\lambda|u_p|^p d\mathcal H^{N-1}
\qquad B^{p'}=|\Omega|+\int_{\partial\Omega}\lambda\, d\mathcal H^{N-1}\,,\]
one can apply Proposition \ref{holder} in order to obtain
\begin{equation*}\label{NuovoHolder}
	A^p\le M(f,g,\lambda)\left[\int_\Omega|\nabla u_p|\, dx+\int_{\partial\Omega}\lambda|u_p|\,  d\mathcal H^{N-1}\right]\le M(f,g,\lambda) AB\,,
\end{equation*}
so that
\[A^{p-1}\le M(f,g,\lambda) B.\]
Hence,
\[\left[\int_\Omega|\nabla u_p|^pdx+\int_{\partial\Omega}\lambda|u_p|^p d\mathcal H^{N-1}\right]^{\frac{p-1}{p}}\le M(f,g,\lambda) \left[|\Omega|+\int_{\partial\Omega}\lambda\, d\mathcal H^{N-1}\right]^{\frac{p-1}{p}},\]
from which we deduce
\begin{equation}\label{stimaprinc}
	\int_\Omega|\nabla u_p|^pdx+\int_{\partial\Omega}\lambda|u_p|^p d\mathcal H^{N-1}\le M(f,g,\lambda)^{\frac{p}{p-1}} \left[|\Omega|+\int_{\partial\Omega}\lambda\, d\mathcal H^{N-1}\right].
\end{equation}
Then it follows from Proposition \ref{holder} and from \eqref{stimaprinc} that we get
\begin{equation}\label{StimaPrincipale}
	\begin{aligned}
	\| u_p\|_\lambda&=\int_\Omega|\nabla u_p|\, dx+\int_{\partial\Omega}\lambda|u_p|\,  d\mathcal H^{N-1}\\
	&\le \left[\int_\Omega|\nabla u_p|^pdx+\int_{\partial\Omega}\lambda|u_p|^p d\mathcal H^{N-1}\right]^{\frac{1}{p}}
	\left[|\Omega|+\int_{\partial\Omega}\lambda\, d\mathcal H^{N-1}\right]^{\frac{1}{p'}}\\
	&\le M(f,g,\lambda)^{\frac1{p-1}}\left[|\Omega|+\int_{\partial\Omega}\lambda\, d\mathcal H^{N-1}\right].
	\end{aligned}
\end{equation}
If $M(f,g,\lambda) \le 1$ the previous estimate reads as
\[\| u_p\|_\lambda\le |\Omega|+\int_{\partial\Omega}\lambda\, d\mathcal H^{N-1}.\]
Then standard compactness arguments hold and there exists a function $u$ such that, up to subsequences, $u_p$ converges to $u$ *-weakly in $BV(\Omega)$ as $p\to 1^+$.

\medskip

Moreover the same estimate \eqref{StimaPrincipale}, if $M(f,g,\lambda)< 1$, guarantees that
\[\lim_{p\to1^+}\|u_p\|_\lambda=0,\]
which means that $u_p$ goes to zero almost everywhere in $\Omega$ as $p\to1^+$.	
\end{proof}

Let us show now that $|\nabla u_p|^{p-2}\nabla u_p$ and $|u_p|^{p-2}u_p$
weakly converges to some functions in $\Omega$ and on $\partial\Omega$ as $p\to 1^+$. Next theorem identifies these objects.

\begin{lemma}\label{teorema1}
	Under the assumptions of Lemma \ref{lemmastimafond}, let $u_p$ be the solution to problem \eqref{pbmain}.
	Then there exist $\z\in L^\infty(\Omega; \R^N)$ and $\beta\in L^s(\partial\Omega,\lambda)$ for every $s<\infty$ such that $\beta\chi_{\{\lambda>0\}}\in L^\infty(\partial\Omega)$ satisfying, up to subsequences, the following convergences
	\begin{equation}\label{conv1}
		|\nabla u_p|^{p-2}\nabla u_p \rightharpoonup \z\qquad \hbox{weakly in }L^s(\Omega; \R^N) \hbox{ for every }1\le s<\infty,
	\end{equation}
	\begin{equation}\label{conv2}
		|u_p|^{p-2} u_p \rightharpoonup \beta\qquad \hbox{weakly in }L^s(\partial\Omega, \lambda) \hbox{ for every }1\le s<\infty.
	\end{equation}
	Moreover, the following identities hold
	\begin{equation}\label{ide1}
		\max\{\|\z\|_\infty , \|\beta\chi_{\{\lambda>0\}}\|_\infty\}=M(f,g,\lambda)
	\end{equation}
	\begin{equation}\label{ide2}
		-\Div\z=f\qquad \hbox{in }\mathcal D'(\Omega)
	\end{equation}
	\begin{equation}\label{ide3}
		[\z,\nu]+\lambda \beta=g\qquad\mathcal H^{N-1}\hbox{--a.e. on }\partial\Omega
	\end{equation}
\end{lemma}

\begin{proof}
	It follows from Lemma \ref{lemmastimafond} that it holds
	\[\int_\Omega |\nabla u_p|^{p}dx+\int_{\partial\Omega}\lambda|u_p|^{p} d\mathcal H^{N-1}\le M(f,g,\lambda)^{\frac p{p-1}} \Lambda\]
	where $\Lambda=|\Omega|+\int_{\partial\Omega}\lambda\, d\mathcal H^{N-1}$.
	Let us now fix $s\in (1,\infty)$ and consider $\displaystyle 1<p<\frac s{s-1}$. By Proposition \ref{holder} below, it yields
	\begin{equation}\label{stima1}
		\begin{aligned}
		&\left[\int_\Omega |\nabla u_p|^{(p-1)s}dx+\int_{\partial\Omega}\lambda|u_p|^{(p-1)s} d\mathcal H^{N-1}\right]^{\frac1s}\\
		&\le \left[\int_\Omega |\nabla u_p|^{p}dx+\int_{\partial\Omega}\lambda|u_p|^{p} d\mathcal H^{N-1}\right]^{\frac{p-1}p}\Lambda^{\frac1s-\frac{p-1}p}
		\le M(f,g,\lambda)\Lambda^{\frac1s}
		\end{aligned}
	\end{equation}
	from where we infer that this family is bounded. Thus, up to subsequences, there exist $\z_s\in L^s(\Omega; \R^N)$ and $\beta_s\in L^s(\partial\Omega, \lambda)$ satisfying
	\[|\nabla u_p|^{p-2}\nabla u_p \rightharpoonup \z_s\qquad \hbox{weakly in }L^s(\Omega; \R^N)
	\]
	and
	\[
	|u_p|^{p-2} u_p \rightharpoonup \beta_s\qquad \hbox{weakly in }L^s(\partial\Omega, \lambda) \]
	Since these facts hold for every $s$, two diagonal procedures allow us to find  $\z\in L^s(\Omega; \R^N)$ and $\beta\in L^s(\Omega,\lambda)$ for all $s\in(1,\infty)$, and satisfying \eqref{conv1} and \eqref{conv2}.
	
	Moreover, having in mind the lower semicontinuity of the $s$--norm with respect to the weak convergence, we may let $p$ go to 1 in \eqref{stima1}; it yields
	\[\left[\int_\Omega |\z|^s\, dx+\int_{\partial\Omega}\lambda|\beta|^{s}\, d\mathcal H^{N-1}\right]^{\frac1s}\\
	\le M(f,g,\lambda)\Lambda^{\frac1s}\]
	for every $s\in (1,\infty)$.
	Thanks to Proposition \ref{bounded} below, we deduce that $\z\in L^\infty(\Omega; \R^N)$ and $\beta\chi_{\{\lambda>0\}}\in L^\infty(\partial\Omega)$. In addition, we may take the limit as $s$ tends to $\infty$ and obtain
	\[\max\{\|\z\|_\infty, \|\beta\chi_{\{\lambda>0\}}\|_\infty\}\le M(f,g,\lambda).\]
	Now let us show the reverse inequality in order to deduce \eqref{ide1}; to this aim we take $v\in W^{1,2}(\Omega)$ as test function in \eqref{pbmain} (with $1<p<2$) to get
	\begin{equation*}
		\int_\Omega fv\, dx+\int_{\partial\Omega} gv\, d\mathcal H^{N-1}=\int_\Omega|\nabla u_p|^{p-2}\nabla u_p\cdot \nabla v\, dx+\int_{\partial\Omega}\lambda |u_p|^{p-2}u_p v\, d\mathcal H^{N-1}.
	\end{equation*}
	Letting $p$ go to 1, we deduce
	\begin{equation*}
		\begin{aligned}
		\int_\Omega fv\, dx+\int_{\partial\Omega} gv\, d\mathcal H^{N-1} &=\int_\Omega\z\cdot \nabla v\, dx+\int_{\partial\Omega}\lambda \beta v\, d\mathcal H^{N-1}\\
		&\le \|\z\|_\infty\int_\Omega|\nabla v|\, dx+\|\beta\chi_{\{\lambda>0\}}\|_\infty\int_{\partial\Omega}\lambda |v|\, d\mathcal H^{N-1}\\
		&\le \max\{\|\z\|_\infty,\|\beta\chi_{\{\lambda>0\}}\|_\infty\}\|v\|_\lambda.
		\end{aligned}	
	\end{equation*}
	By density, it yields
	\[\int_\Omega fv\, dx+\int_{\partial\Omega} gv\, d\mathcal H^{N-1}\le \max\{\|\z\|_\infty,\|\beta\chi_{\{\lambda>0\}}\|_\infty\}\|v\|_\lambda\]
	for every $v\in W^{1,1}(\Omega)$. Therefore,
	\[M(f,g,\lambda)\le \max\{\|\z\|_\infty,\|\beta\chi_{\{\lambda>0\}}\|_\infty\},\]
	which gives \eqref{ide1}.
	
	\medskip
	
	The validity of \eqref{ide2} simply follows by taking $\varphi\in C_0^\infty(\Omega)$ as test function in \eqref{pbmain} and letting $p\to 1^+$.

	Now for $1<p<2$, we choose $v\in W^{1,2}(\Omega)$ as test function in \eqref{pbmain} obtaining:
	\begin{equation*}
		\int_\Omega fv\, dx+\int_{\partial\Omega} gv\, d \mathcal H^{N-1}=\int_\Omega|\nabla u_p|^{p-2}\nabla u_p\cdot \nabla v\, dx+\int_{\partial\Omega}\lambda |u_p|^{p-2}u_p v\, d\mathcal H^{N-1}.
	\end{equation*}
	Letting $p\to 1^+$, it yields
	\begin{equation*}
		\int_\Omega fv\, dx+\int_{\partial\Omega} gv\, \mathcal H^{N-1}=\int_\Omega \z\cdot \nabla v\, dx+\int_{\partial\Omega}\lambda \beta v\, d\mathcal H^{N-1}
	\end{equation*}
	for every $v\in W^{1,2}(\Omega)$. This equality can be extended to every $v\in W^{1,1}(\Omega)$ by density. Using \eqref{ide2} and Green's formula, we deduce
	\begin{equation*}
		\int_{\partial\Omega} gv\, \mathcal H^{N-1}=\int_{\partial\Omega}v[\z,\nu]\, d\mathcal H^{N-1}+\int_{\partial\Omega}\lambda \beta v\, d\mathcal H^{N-1}
	\end{equation*}
	for all $v\in W^{1,1}(\Omega)$, wherewith it holds for every $v\in L^1(\partial\Omega)$. Thus, we have obtained \eqref{ide3}.
\end{proof}

The following lemma focuses on the behaviour of the objects studied in the previous lemma when the family $u_p$ is truncated at a certain level. This will be useful in the next Lemma \ref{lemblowup}.

\begin{lemma}\label{teorema2}
	Under the assumptions of Lemma \ref{lemmastimafond}, let $u_p$ be the solution to problem \eqref{pbmain}.
	For each $k>0$ there exist $\z_k\in L^\infty(\Omega; \R^N)$ and $\beta_k$ such that $\beta_k\chi_{\{\lambda>0\}}\in L^\infty(\partial\Omega)$ satisfying $\|\z_k\|_\infty\le1$, $\|\beta_k\chi_{\{\lambda>0\}}\|_\infty\le1$ and, up to subsequences, the following convergences hold
	\begin{equation*}\label{conv3}
		|\nabla u_p|^{p-2}\nabla u_p \chi_{\{|u_p|<k\}} \rightharpoonup \z_k\qquad \hbox{weakly in }L^s(\Omega; \R^N) \hbox{ for every }1\le s<\infty,
	\end{equation*}
	\begin{equation*}\label{conv4}
		|u_p|^{p-2} u_p \chi_{\{|u_p|<k\}}\rightharpoonup \beta_k\qquad \hbox{weakly in }L^s(\partial\Omega, \lambda) \hbox{ for every }1\le s<\infty.
	\end{equation*}
\end{lemma}

\begin{proof}
	For each $k>0$, we take $T_k(u_p)$ as test function in \eqref{pbmain}, it yields
	\[\int_{\Omega} |\nabla T_k(u_p)|^pdx+\int_{\partial\Omega}\lambda |u_p|^{p-1}|T_k(u_p)|\, d\mathcal H^{N-1}=\int_\Omega fT_k(u_p)\, dx+\int_{\partial\Omega}gT_k(u_p)\, d\mathcal H^{N-1},\]
	from where we get the estimate
	\begin{equation*}\label{est1}
		\int_{\Omega} |\nabla T_k(u_p)|^pdx\le k\left(\int_\Omega |f|\, dx+\int_{\partial\Omega}|g|\, d\mathcal H^{N-1}\right).
	\end{equation*}
	Given $s<\frac{p}{p-1}$, H\"older's inequality implies
	\begin{equation}\label{est2}
		\begin{aligned}
		\left[\int_\Omega |\nabla u_p|^{(p-1)s}\chi_{\{|u_p|<k\}}\, dx\right]^{\frac1s}
		&\le \left[\int_\Omega |\nabla u_p|^{p}\chi_{\{|u_p|<k\}}\, dx\right]^{\frac{p-1}p}|\Omega|^{\frac1s-\frac{p-1}p}\\
		&\le k^{\frac{p-1}p}\left(\int_\Omega |f|\, dx+\int_{\partial\Omega}|g|\, d\mathcal H^{N-1}\right)^{\frac{p-1}p}|\Omega|^{\frac1s-\frac{p-1}p}.
		\end{aligned}	
\end{equation}
	Hence, the family $|\nabla u_p|^{p-2}\nabla u_p\chi_{\{|u_p|<k\}}$ is bounded in $L^s(\Omega; \R^N)$ for all $s\in (1,\infty)$. By the same procedure used in Lemma \ref{teorema1}, there exist $\z_k\in L^s(\Omega; \R^N)$ and a subsequence (not relabeled) such that
	\[|\nabla u_p|^{p-2}\nabla u_p\chi_{\{|u_p|<k\}} \rightharpoonup \z_k\]
	for all $s\in (1,\infty)$. Going back to \eqref{est2} and letting $p$ go to 1, the lower semicontinuity of the $s$--norm with respect to the weak convergence gives
	\[\left[\int_\Omega |\z_k|^s\, dx\right]^{\frac1s}\le |\Omega|^{\frac1s}\]
	for all $s\in (1,\infty)$. Therefore, $\z_k\in L^\infty(\Omega; \R^N)$ and $\|\z_k\|_\infty\le 1$.
	
	On the other hand, it follows from $|u_p|^{p-1}\chi_{\{|u_p|<k\}}\le k^{p-1}$ that, up to subsequences,
	\[|u_p|^{p-2}u_p\chi_{\{|u_p|<k\}} {\buildrel * \over \rightharpoonup} \beta_k\qquad \hbox{*-weakly in } L^\infty(\partial\Omega,\lambda)\]
	for certain $\beta_k\in  L^\infty(\partial\Omega,\lambda)$ that satisfies $\|\beta_k\chi_{\{\lambda>0\}}\|_\infty\le1$.
\end{proof}

Here we deal with the case $M(f,g,\lambda)> 1$; in particular we show that $u_p$ blows up on a set of positive measure.

\begin{lemma}\label{lemblowup}
	Under the assumptions of Lemma \ref{lemmastimafond}, let $u_p$ be the solution to problem \eqref{pbmain}. If $M(f,g,\lambda)>1$, then $u_p$ converges almost everywhere in $\Omega$ as $p\to 1^+$ to a function $u$ such that $|u|=+\infty$ either on a subset of $\Omega$ of positive Lebesgue measure or on a subset of $\partial\Omega$ of positive $\mathcal H^{N-1}$ measure.
	As a consequence, $u\notin BV(\Omega)$.
\end{lemma}
\begin{proof}
Firstly one can show that $u_p$ converges almost everywhere in $\Omega$ to a function $u$ as $p\to 1^+$ using arguments similar to the ones of Step 2 of \cite{mst2}.
It follows from the pointwise convergence $u_p\to u$ as $p\to 1^+$ that
\[\chi_{\{|u_p|<k\}}\to \chi_{\{|u|<k\}}\qquad\hbox{strongly in } L^r(\Omega)\ \forall r\in(1,\infty)\]
up to a countable set of $k>0$. So, for almost all $k>0$, it follows from Lemmas \ref{teorema1} and \ref{teorema2} that we have
\[\z_k=\z\chi_{\{|u|<k\}}\qquad \hbox{and}\qquad \beta_k \chi_{\{\lambda>0\}}=\beta \chi_{\{\{\lambda>0\} \cap \{|u|<k\}\}}\]
Thus,
conditions $\|\z_k\|_\infty\le1$ and $\|\beta_k\chi_{\{\lambda>0\}}\|_\infty\le1$ for all $k>0$ imply
\[\|\z\chi_{\{|u|<\infty\}}\|_\infty\le 1\qquad \hbox{and}\qquad \|\beta \chi_{\{\{\lambda>0\}\cap\{|u|<\infty\}\}}\|_\infty\le1.\]
Having in mind \eqref{ide1}, the result follows.
\end{proof}
Finally we can gather the previous results to give the proof of the main result of the current section.
\begin{proof}[Proof of Theorem \ref{mainteo}]
	The proof follows from Lemmas \ref{lemmastimafond} and \ref{lemblowup}.
\end{proof}

\section{The limit problem}
\label{sec:limitproblem}
Here we are interested into the study of the limit problem for \eqref{pbmain} as $p\to 1^+$. In particular we first deal with the case with $\Omega$ regular enough. Later and under some assumptions on the data, we treat the case where $\Omega$ has Lipschitz boundary.  
\\
Thus we are studying the existence of a solution to
\begin{equation}
	\label{pb1}
	\begin{cases}
		\displaystyle -\Delta_1 u = f & \text{ in }\Omega,\\
		\displaystyle [\z,\nu] +\lambda \sg(u) = g & \text{ on } \partial\Omega.
	\end{cases}
\end{equation}
Let us stress that the sign function needs to be intended as a multivalued function which is  $\sg(u)=[-1,1]$ when $u=0$. 
Then let us specify the notion of solution we adopt for problem \eqref{pb1}.

\begin{defin}\label{def}
	A function $u\in BV(\Omega)$ is a solution to \eqref{pb1} if there exists $\z\in L^\infty(\Omega,\R^N)$ with $||\z||_{\infty}\le 1$ such that
	\begin{align}
		&-\operatorname{div}\z = f \ \ \text{as measures in }\Omega, \label{def_distrp=1}
		\\
		&(\z,Du)=|Du| \label{def_zp=1} \ \ \ \ \text{as measures in } \Omega,
		\\
		&	[\z,\nu]+\lambda \beta=g\label{def_bordop=1}\ \ \ \text{for  $\mathcal{H}^{N-1}$-a.e. } x \in \partial\Omega,
	\end{align}
\end{defin}
where $\beta$ is a measurable function such that $\|\beta \chi_{\{\lambda>0\}}\|_{\infty} \le 1$ 
and
\begin{equation}\label{def_bordo2p=1} (\lambda\beta -g) \in   T_1(\lambda \sg(u)-g)  \ \ \ \text{for  $\mathcal{H}^{N-1}$-a.e. } x \in \partial\Omega.
\end{equation}
\begin{remark}\label{remdef}
	The notion of solution given by Definition \ref{def} is nowadays classical in the context of $1$-Laplace operator. Equation \eqref{def_zp=1} is how $\z$ plays the role of the quotient $|D u|^{-1}D u$, which, jointly with \eqref{def_distrp=1}, formally represents the equation in problem \eqref{pb1}. Equations \eqref{def_bordop=1} and \eqref{def_bordo2p=1} deserve a particular attention. It is clear that if $|\lambda \sg(u)-g|\le 1$ then \eqref{def_bordo2p=1} means $\beta \in  \sg(u)$ in $\{\lambda>0\}$ which is what one clearly expect as for the boundary equation in \eqref{pb1}. Otherwise, if $|\lambda \sg(u)-g| > 1$, then \eqref{def_bordo2p=1} in \eqref{def_bordop=1}  simply means that $|[\z,\nu]|$ is forced to be highest possible.
\end{remark}

\subsection{The case $\de\Omega\in C^{1}$}

In this section $\Omega$ is a  bounded open set of $\R^{N}$ with $C^{1}$ boundary. 

\medskip

The main result of this section is the following:

\begin{theorem}\label{existence}
Let $f\in L^{N,\infty}(\Omega)$, $g\in L^\infty(\partial\Omega)$ and let $\lambda \in L^{\infty}(\partial\Omega)$ be nonnegative but not identically null. If $M(f,g,\lambda)\le 1$, then there exists a solution to \eqref{pb1}.
\end{theorem}

\begin{remark}
	One can wonder if the solution found in Theorem \ref{existence} is actually the unique one. In the context of the $1$-Laplace operator this is often a  delicate issue. Let us stress that indeed for problem \ref{pb1} one can not expect uniqueness of solutions in the sense of Definition \ref{def}. Indeed, let $F$ be an increasing function such that $F(0)=0$.  It is now simple to convince that if $u$ is a solution to \eqref{pb1} then $F(u)$ is a solution itself to the same problem.
\end{remark}

Clearly we will prove Theorem \ref{existence} by means of approximation through problems \eqref{pbmain} and using the information already gained on $u_p$.
Henceforth $\z$ and $\beta$ are the ones found in Lemmas \ref{teorema1} and \ref{teorema2} respectively.

\medskip

Hence we just need to show the identification of both $\z$ and $\beta$ by proving \eqref{def_zp=1} and \eqref{def_bordo2p=1}.

\medskip

We start by proving the identification of $\beta$; we first show that the assumption on $M$ can be read as an assumption connecting  $\lambda$ and $g$ in an explicit way.

\begin{lemma}\label{desug1}
	Under the assumptions of Theorem \ref{existence} let $u_p$ be a solution of \eqref{pbmain}. If $M(f,g,\lambda)\le 1$, then $|g|\le\lambda+1$. As a consequence,
	\begin{equation}\label{des}
	T_1(\lambda \sg (r)-g)r\le \lambda|r|-gr
	\end{equation}
\end{lemma}
holds  for all $r\in\R$.
\begin{proof}
	It follows from Lemma \ref{teorema1} that if $M(f,g,\lambda)\le 1$ then $\|\z\|_\infty\le1$ and $\|\beta\chi_{\{\lambda>0\}}\|_\infty\le1$, so that $-1\le [\z,\nu]\le1$ and $-1\le \beta\chi_{\{\lambda>0\}}\le1$. These facts and the identity $[\z,\nu]+\lambda\beta=g$ yield the desired inequality. Indeed,
	\[\lambda-g\ge\lambda\beta-g=-[\z,\nu]\ge-1\]
	\[-\lambda-g\le \lambda\beta-g=-[\z,\nu]\le1\]
	wherewith $g\le\lambda+1$ and $-g\le\lambda+1$ hold.
	
	\medskip
	
	It is enough to analyze two possibilities since \eqref{des} trivially holds when $r=0$.
	
	If $r>0$, since we have already proven $-1\le \lambda -g$, then $T_1(\lambda -g)r\le \lambda r-gr$.
	
	If $r<0$, since one has $-1\le \lambda +g$ , then
	$-T_1(\lambda +g)r\le -\lambda r-gr$, that is $T_1(-\lambda -g)r\le \lambda|r|-gr$.
\end{proof}

The previous lemma allows us to prove the following result.
\begin{lemma}\label{lemmabordo}
		Under the assumptions of Theorem \ref{existence} let $u_p$ be a solution of \eqref{pbmain} and let $\z$ and $\beta$  be the vector field and the function found in Lemma \ref{teorema1}. Then it holds
		\begin{equation*}
			u([\z,\nu]+ T_1(\lambda\sg u-g))= 0\qquad\mathcal H^{N-1}\hbox{--a.e. on }\partial\Omega.
		\end{equation*}
	In particular it holds \eqref{def_bordo2p=1}.
\end{lemma}

\begin{proof}
	Let us take $T_{k}(u_p)$ as a test function in \eqref{pbmain} obtaining that
		$$
		\int_\Omega |\nabla T_{k}(u_p)|^{p}\, dx + \int_{\partial\Omega}\lambda |T_{k}(u_p)|^{p} d\mathcal H^{N-1} = \int_\Omega f T_{k}(u_p)\, dx + \int_{\partial\Omega}g T_{k}(u_p) d\mathcal H^{N-1},
		$$
		which, applying the Young inequality, implies that
		\begin{equation}\label{bordo0}
			\begin{aligned}
			&\int_\Omega |\nabla T_{k}(u_p)|\, dx + \int_{\partial\Omega}(\lambda|T_{k}(u_p)|-g T_{k}(u_p))  d\mathcal H^{N-1} \\ &\le  \int_\Omega f T_{k}(u_p)\, dx + \frac{p-1}{p}|\Omega| + \frac{p-1}{p}\int_{\partial\Omega} \lambda d\mathcal H^{N-1}.
			\end{aligned}
		\end{equation}
		Owing to Lemma \ref{desug1}, \eqref{bordo0} becomes
		\begin{equation}\label{bordo1}
			\begin{aligned}
			&\int_\Omega |\nabla T_{k}(u_p)|\, dx + \int_{\partial\Omega}T_1(\lambda\sg(u_p)-g) T_{k}(u_p)\,  d\mathcal H^{N-1} \\
			&\le  \int_\Omega f T_{k}(u_p)\, dx + \frac{p-1}{p}\left[|\Omega| + \int_{\partial\Omega} \lambda d\mathcal H^{N-1}\right].
			\end{aligned}
		\end{equation}
		Notice that the left hand side of \eqref{bordo1}  is lower semicontinuous with respect to the $L^1$-convergence as $p\to 1^+$ thanks to Proposition $1.2$ of \cite{modica}. Hence, taking $p\to 1^+$ in \eqref{bordo1}, one yields to
		\begin{equation}\label{bordo2}
			\int_\Omega |D T_{k}(u)| + \int_{\partial\Omega} T_1(\lambda\sg u-g) T_{k}(u)\, d\mathcal H^{N-1} \le  \int_\Omega f T_{k}(u)\, dx.
		\end{equation}
		Now since it follows from Lemma \ref{teorema1} that $-\Div\z = f$, from \eqref{bordo2} one deduces that
		\begin{equation*}
			\begin{aligned}
			\int_\Omega |D T_{k}(u)| + \int_{\partial\Omega} T_1(\lambda\sg u-g) T_{k}(u)\, d\mathcal H^{N-1} &\le  -\int_\Omega \Div\z\, T_{k}(u)
			\\
			&= \int_\Omega (\z, DT_{k}(u)) - \int_{\partial\Omega} T_{k}(u)[\z,\nu]\,d\mathcal H^{N-1},
			\end{aligned}
		\end{equation*}
		where the last equality follows from an application of the Green formula. Now observe that $(\z,DT_{k}(u)) \le |DT_{k}(u)|$ as measures since $\|\z\|_{\infty}\le 1$; then one gets
		\begin{equation}\label{bordo4}
			\int_{\partial\Omega} (T_1(\lambda\sg u-g) + [\z,\nu]) T_{k}(u)\, d\mathcal H^{N-1} \le 0.
		\end{equation}
		Now observe that $\left(T_1(\lambda\sg u-g) + [\z,\nu]\right)$ has the same sign of $u$ for $\mathcal H^{N-1}$--almost every point on $\partial\Omega$. Indeed, assume first that $x\in\partial\Omega$ satisfies $u(x)>0$. Then $\lambda(x)-g(x) \ge -1$ by Lemma \ref{desug1}. If $\lambda(x)-g(x)\le1$, then Lemma \ref{teorema1} gives $\lambda-g +[\z,\nu]=\lambda-g + g-\lambda \beta=\lambda-\lambda\beta \ge 0$ since $|\beta\chi_{\{\lambda>0\}}|\le 1$. Otherwise let  $x$ be such that $\lambda(x)-g(x) > 1$ then $1+[\z,\nu] \ge 0$ since $|[\z,\nu]|\le 1$. A similar argument holds when $u(x)<0$.
		
		Thus, \eqref{bordo4} implies that
		$(T_1(\lambda\sg u-g) + [\z,\nu]) u = 0$ $\mathcal H^{N-1}$--almost everywhere on $\partial \Omega$. Moreover since  $[\z,\nu] = g-\lambda \beta$ it follows \eqref{def_bordo2p=1}.
	\end{proof}

Now we focus on proving \eqref{def_zp=1}.

\begin{lemma}\label{lemmaidentificaz}
	Under the assumptions of Theorem \ref{existence} let $u_p$ be a solution of \eqref{pbmain} and let $\z$ be the vector field found in Lemma \ref{teorema1}. Then it holds
	\begin{equation*}
		(\z,Du)=|Du|  \ \ \ \ \text{as measures in } \Omega.
	\end{equation*}
\end{lemma}

\begin{proof}
	Let us take $T_k(u_p)\varphi$ ($k> 0,\; 0\le \varphi\in C^1_c(\Omega)$) as a test function in \eqref{pbmain} yielding to
	\begin{equation*}
	\int_\Omega |\nabla T_k(u_p)|^p\varphi \, dx + \int_\Omega T_k(u_p) |\nabla u_p|^{p-2}\nabla u_p\cdot \nabla \varphi \, dx = \int_\Omega f T_k(u_p)\varphi\, dx,
	\end{equation*}
	which, from an application of the Young inequality, implies
	\begin{equation*}
	\int_\Omega |\nabla T_k(u_p)|\varphi \, dx + \int_\Omega T_k(u_p) |\nabla u_p|^{p-2}\nabla u_p\cdot \nabla \varphi \, dx \le \int_\Omega f T_k(u_p)\varphi\, dx + \frac{p-1}{p}\int_\Omega \varphi \, dx.
	\end{equation*}
	By taking $p\to 1^+$ in the previous inequality, one obtains that
	\begin{equation*}
		\int_\Omega |D T_k(u)|\varphi + \int_\Omega T_{k}(u) \z\cdot \nabla \varphi \, dx \le \int_\Omega f T_k(u)\varphi\, dx.
	\end{equation*}	
Hence, letting $k\to +\infty$, 	
		\begin{equation*}
		\int_\Omega |D u|\varphi + \int_\Omega u \z\cdot \nabla \varphi \, dx \le \int_\Omega f u\varphi\, dx.
	\end{equation*}	
	Now, recalling that $-\Div \z = f$ one has that
		\begin{equation*}
		\int_\Omega |D u|\varphi \le - \int_\Omega u \z\cdot \varphi \, dx  -\int_\Omega \Div \z u\varphi\, dx = \int_\Omega (\z, D u)\varphi.
	\end{equation*}
	 This concludes the proof being the reverse inequality trivial  since $||\z||_{\infty}\le 1$.
\end{proof}

\begin{proof}[Proof of Theorem \ref{existence}]
	Let $u_p$ be a solution to \eqref{pbmain}. Then it follows from Lemma \ref{teorema1} that there exist $u\in BV(\Omega)$ and $\z\in X(\Omega)$ with $||\z||_{\infty}\le 1$ such that \eqref{def_distrp=1} and \eqref{def_bordop=1} hold. Moreover Lemmas \ref{lemmaidentificaz} and \ref{lemmabordo} give that \eqref{def_zp=1} and \eqref{def_bordo2p=1} hold respectively. This concludes the proof.
\end{proof}

\subsection{The case $\de\Omega$ Lipschitz}

In the previous subsection we required that $\Omega$ has $C^1$ boundary. This fact is due to the application in Lemma \ref{lemmabordo} of Modica's semicontinuity result that needs this hypothesis. Nevertheless, as Modica himself points out, certain functionals are lower semicontinuous with respect to the $L^1$-convergence even when the Lipschitz-continuous setting is considered.

We prove the following result.
\begin{lemma}\label{low-sem1}
Let $H\>:\> BV(\Omega)\to\R$ be a functional defined as
\[H(u)=\int_\Omega|Du|+\int_{\partial\Omega}\psi(x)|u|\, d\mathcal H^{N-1}\]
where $\psi\in L^\infty(\partial\Omega)$ satisfies $0\le \psi\le1$.

Then $H$ is lower semicontinuous with respect to the $L^1$-convergence.
\end{lemma}

\begin{proof}
We first choose an open bounded set $\Omega'$ containing $\overline\Omega$.
Given $\psi\in L^\infty(\partial\Omega)$, we may find $\phi_1\in C^1(\Omega)\cap W^{1,1}(\Omega)$ such that $\phi_1\big|_{\partial\Omega}=\psi$ and $0\le\phi_1\le1$. We may also consider $\phi_2\in C^1(\Omega'\backslash \overline\Omega)\cap W^{1,1}(\Omega'\backslash \overline\Omega)$ such that $\phi_2\big|_{\partial\Omega}=\psi$ and $0\le\phi_2\le1$. Finally define the following continuous extension of $\psi$:
\begin{equation*}\label{exten}
\varphi(x)=\left\{\begin{array}{ll}
\phi_1(x) &\hbox{ if }x\in\Omega\\[2mm]
\phi_2(x) &\hbox{ if }x\in\Omega'\backslash \overline\Omega\,.
\end{array}\right.
\end{equation*}

We next claim that each  $u\in BV(\Omega)$ satisfies
\begin{equation*}\label{iden}
  \int_{\Omega} \phi_1|Du|+\int_{\partial\Omega}\psi|u|\, d\mathcal H^{N-1} 
  =\sup\left\{\int_{\Omega'} u\, \Div(\varphi F)\, dx\>:\>F\in C_0^1(\Omega')^N\ \ \|F\|_\infty\le1\right\}\,,
\end{equation*}
where $u$ is extended to  $BV(\Omega')$ by defining $u=0$ in $\Omega'\backslash \overline\Omega$.

An inequality is obvious since Green's formula implies
\begin{equation*}
	\begin{aligned}
\int_{\Omega'} u\, \Div(\varphi F)\, dx &=\int_{\Omega} u\, \Div(\phi_1 F)\, dx\\
&=-\int_{\Omega} \phi_1 F\cdot Du+\int_{\partial\Omega}\psi u [F, \nu]\, d\mathcal H^{N-1}\\
&\le \int_{\Omega} \phi_1|Du|+\int_{\partial\Omega}\psi|u|\, d\mathcal H^{N-1}
	\end{aligned}
\end{equation*}
holds for all $F\in C_0^1(\Omega')^N$ such that $\|F\|_\infty\le1$.

To check the reverse inequality, we consider in $C_0^1(\Omega')^N$ the linear map given by
\[
L(F)=\int_{\Omega'} u\, \Div(\varphi F)\, dx\,.
\]

Notice that
\begin{equation*}
	\begin{aligned}
|L(F)|=\left|\int_{\Omega'} u\, \Div(\varphi F)\, dx\right|=\left|-\int_{\Omega} \phi_1 F\cdot Du+\int_{\partial\Omega}\psi u [F, \nu]\, d\mathcal H^{N-1}\right|\\
\le \|F\|_\infty\left[\int_{\Omega} \phi_1|Du|+\int_{\partial\Omega}\psi|u|\, d\mathcal H^{N-1}\right]\,.
	\end{aligned}
\end{equation*}
From this inequality we deduce that $L$ can be extended by density to a linear and continuous map in $C_0(\Omega')^N$ whose norm satisfies
\[
\|L\|\le \int_{\Omega} \phi_1|Du|+\int_{\partial\Omega}\psi|u|\, d\mathcal H^{N-1}\,.
\]
Applying the Riesz representation Theorem, there exists a Radon measure $\mu$ on $\Omega'$ such that $L(F)=\int_{\Omega'} F\cdot\mu$ for every $F \in C_0(\Omega')^N$ and its total variation is $\int_{\Omega'}|\mu|=\|L\|$. Thus,
\[
\int_{\Omega'} F\cdot \mu=L(F)=\int_{\Omega'} u\, \Div(\varphi F)\, dx=-\int_\Omega \phi_1 F\cdot Du+\int_{\partial\Omega}u\psi [F, \nu]\, d\mathcal H^{N-1}
\]
for all $F\in C_0^1(\Omega')^N$.  We deduce that
\[
\int_{\Omega} \phi_1|Du|+\int_{\partial\Omega}\psi|u|\, d\mathcal H^{N-1}=\int_{\Omega'} |\mu|=\sup\left\{L( F)\>:\> F\in C_0(\Omega')^N\ \ \|F\|_\infty\le1\right\}\,.
\]
By density, we conclude that
\[\int_{\Omega} \phi_1|Du|+\int_{\partial\Omega}\psi|u|\, d\mathcal H^{N-1}=\sup\left\{L( F)\>:\> F\in C_0^1(\Omega')^N\ \ \|F\|_\infty\le1\right\}\]
and the claim is proven.

As a straightforward consequence the functional
\[u\mapsto \int_{\Omega} \phi_1|Du|+\int_{\partial\Omega}\psi|u|\, d\mathcal H^{N-1}\]
is lower semicontinuous with respect to the $L^1$--convergence. Therefore,
\[H(u)=\int_{\Omega} (1-\phi_1)|Du|+\int_{\Omega} \phi_1|Du|+\int_{\partial\Omega}\psi|u|\, d\mathcal H^{N-1}\]
is the sum of two lower semicontinuous functionals, so that Lemma is proven.
\end{proof}

The previous lemma can be applied to the functional
\begin{equation}\label{low-sem0}
I(u)=\int_\Omega |D u| + \int_{\partial\Omega}T_1(\lambda\sg(u)-g) u\,  d\mathcal H^{N-1},
\quad u\in BV(\Omega)\,,
\end{equation}
as shown in Proposition \ref{low-sem2} below. Here we only have to take into account the inequalities $|a^+-b^+|\le |a-b|$ and $|a^--b^-|\le |a-b|$, which hold for all real numbers. 

\begin{proposition}\label{low-sem2}
The functional $I$ defined in \eqref{low-sem0} is lower semicontinuous with respect to the $L^1$--convergence when $|g|\le \lambda$.
\end{proposition}

\begin{proof}
First write $I=I_1+I_2$, where
\[I_1(u)=I(u^+)=\int_\Omega |D u^+|\, dx + \int_{\partial\Omega}T_1(\lambda-g) u^+\,  d\mathcal H^{N-1}\]
and
\[I_2(u)=I(-u^-)=\int_\Omega |D u^-|\, dx + \int_{\partial\Omega}T_1(\lambda+g) u^-\,  d\mathcal H^{N-1}\]

Take a sequence $u_n$ in $BV(\Omega)$ that converges to $u$ strongly in $L^1(\Omega)$. Then
$u_n^+$ converges to $u^+$ and $u_n^-$ converges to $u^-$ as $n\to \infty$, so that Lemma \ref{low-sem1} implies that
\[I_1(u)\le \liminf_{n\to\infty}I_1(u_n)\]
and
\[I_2(u)\le \liminf_{n\to\infty}I_2(u_n)\,.\]
Therefore, its sum $I$ is lower semicontinuous.
\end{proof}

\begin{theorem}
Theorem \ref{existence} holds even if $\Omega$ has Lipschitz boundary in case $|g|\le \lambda$.
\end{theorem}
\begin{proof}
	The only difference with respect to the proof of Theorem \ref{existence} is the use of Proposition \ref{low-sem2} in place of Proposition $1.2$ of \cite{modica}.
\end{proof}

\section{Remarks and examples}
\label{sec:examples}

\subsection{The case with $\lambda \in L^1(\partial\Omega)$}
\label{secL1}

Here we briefly spend a few words for the case of a nonnegative $\lambda \in L^1(\partial\Omega)$. 

\medskip

Indeed, let us stress that, even for $\lambda \in L^1(\partial\Omega)$, the quotient which appears in $M$ is well defined. Nevertheless, now the supremum is taken over all  $u\in W^{1,1}(\Omega)\cap L^1(\partial\Omega, \lambda)\backslash\{0\}$

\medskip

Then if one considers the following approximation scheme
\begin{equation}
	\label{pbL1}
	\begin{cases}
		\displaystyle -\Delta_p u_p = f & \text{ in }\Omega,\\
		\displaystyle |\nabla u_p|^{p-2}\nabla u_p\cdot \nu +\lambda |u_p|^{p-2}u_p = g& \text{ on } \partial\Omega,
	\end{cases}
\end{equation}
where $f\in L^{N,\infty}(\Omega)$, $g\in L^\infty(\partial\Omega)$ and $0\le \lambda \in L^1(\partial\Omega)$ but not identically null, the existence of $u_p\in W^{1,p}(\Omega) \cap L^p(\partial \Omega, \lambda)$ satisfying \eqref{pbL1} follows from the minimization of the following functional
\[
Q(u)=\frac{1}{p}\int_{\Omega}\left|\nabla u\right|^{p}dx +\int_{\partial\Omega}\frac{\lambda}{p}\left|u\right|^{p} \ d\mathcal H^{N-1} - \int_{\partial\Omega}gu \ d\mathcal H^{N-1} - \int_{\Omega}fu \ dx.
\]
Indeed, we consider the space $W^{1,p}(\Omega) \cap L^p(\partial \Omega, \lambda)$ endowed with the norm defined as $\displaystyle \|u\|^p_{p,\lambda}=\int_\Omega|\nabla u|^pdx+\int_{\partial\Omega}\lambda |u|^p\, d\mathcal H^{N-1}$. We remark that  $\|u\|_{p,\lambda}$ is not anymore an equivalent norm to the $W^{1,p}$--norm. By the way one can convince himself that it always holds the inequality
$$\|u\|_{p,\lambda} \ge C\|u\|_{W^{1,p}(\Omega)},$$
which allows to deduce all continuous and compact embeddings which holds for $W^{1,p}(\Omega)$. Since $Q$ can be written as
\[Q(u)=\frac1p\|u\|_{p,\lambda}^p - \int_{\partial\Omega}gu \ d\mathcal H^{N-1} - \int_{\Omega}fu \ dx\,,\]
it follows that these embeddings lead to coercivity. We also deduce from these embeddings that $Q$ is weakly lower semicontinuous. Standard results then yield the desired minimizer.

Any minimizer $u_p$ of the previous functional satisfies that
$$\displaystyle \int_{\Omega}\left|\nabla u_p\right|^{p-2}\nabla u_p\cdot \nabla \varphi \ dx +\int_{\partial\Omega}\lambda\left|u_p\right|^{p-2}u_p\varphi \ d\mathcal H^{N-1} = \int_{\Omega}f\varphi \ dx +  \int_{\partial\Omega}g\varphi \ d\mathcal H^{N-1},$$
where $\varphi \in W^{1,p}(\Omega) \cap L^p(\partial \Omega, \lambda)$.
Therefore $u_p$ itself can be taken as a test function.
Now similar estimates for $\|u_p\|_{\lambda}$ can be obtained. Notice that  $\|u_p\|_{\lambda} \ge C\|u_p\|_{BV(\Omega)}$ but they are not equivalent. With this approach in mind one can show that the results of both Sections \ref{sec_behaviour} and \ref{sec:limitproblem} still hold if $0\le \lambda \in L^1(\Omega)$ (but not identically null) with natural modifications.

\subsection{The radial case}
\label{radialSec}
Here we deal with the case $\Omega$ as a ball of radius $R$ centered at the origin, namely:
\[
\Omega = B_R := \{x\in \mathbb{R}^N: |x|<R\}.
\]
Hence let us consider the following problem
\[
\begin{cases}
	\displaystyle -\Delta_p u_p = \frac{A}{\left|x\right|} &\text{in} \ \Omega,\\
	|\nabla u_p|^{p-2}\nabla u_p \cdot \nu + \lambda u^{p-1}_p=\gamma&\text{on}\ \partial\Omega\,,
\end{cases}
\]
where $A,\lambda$ and $\gamma$ are positive constants, and we are firstly interested in the asymptotic behaviour of $u_p$ as $p\to 1^+$. We explicitly observe that the datum $A/\left|x\right|\in L^{N,\infty}(\Omega)$.

Hence we look for a function $u_p(r)$ ($r=|x|$) satisfying
\[
-\frac{1}{r^{N-1}}\left(r^{N-1} |u_p'(r)|^{p-2} u_p'(r)\right)^{'} = \frac{A}{r},
\]
which gives
\begin{equation*}
	[r^{N-1}(-u_p'(r))^{p-1}]' = \frac{A}{r^{2-N}}	
\end{equation*}
and
\begin{equation}\label{rad1}
-u_{p}'(r)=\left(\frac{A}{N-1}\right)^{\frac{1}{p-1}}.
\end{equation}
Now integrating between $r$ and $R$ (with an abuse of notation) one has
\[
u_p(r) = u_p(R) + \left(\frac{A}{N-1}\right)^{\frac{1}{p-1}}\left(R- r\right),
\]
and since it follows from the boundary condition and from \eqref{rad1} that
\[
u_p(R)=\left\{\frac{1}{\lambda}\left[\frac{A}{N-1}+\gamma\right]\right\}^{\frac{1}{p-1}}
\]
then one also has
\begin{equation*}
	u_p(r) = \left\{\frac{1}{\lambda}\left[\frac{A}{N-1}+\gamma\right]\right\}^{\frac{1}{p-1}}
	+
	\left(\frac{A}{N-1}\right)^{\frac{1}{p-1}}\left(R- r\right).
\end{equation*}
Let us underline that:
\begin{enumerate}
	\item if $A>N-1$, then $u_{p}\to+\infty$ in $\Omega$;
	\item if $A=N-1$, then
	\begin{enumerate}
		\item if $\lambda<1+\gamma$, then $u_{p}\to +\infty$;
		\item if $\lambda=1+\gamma$, then $u_{p}\to 1+(R-r)$;
		\item if $\lambda>1+\gamma$, then $u_{p}\to R-r$;
	\end{enumerate}
	\item if $A<N-1$, then
	\begin{enumerate}
		\item if $\lambda < \frac{A}{N-1}+\gamma$, then $u_{p}\to +\infty$ in $\bar \Omega$;
		\item if $\lambda=\frac{A}{N-1}+\gamma$, then $u_{p}\to 1$;
		\item if $\lambda >\frac{A}{N-1}+\gamma$, then $u_{p}\to 0$.
	\end{enumerate}	
\end{enumerate}

\begin{remark}
Let $\Omega=B_{R}$ and $A,\gamma \ge 0$. A posteriori from Lemma \ref{lemmastimafond} and \ref{lemblowup}, last example assures what follows.
\begin{itemize}[leftmargin=*]
\item[$\textcolor{gray}{\bullet}$] In the cases: $A>N-1$; $A=N-1$ and $\lambda<1+\gamma$; $(N-1)(\lambda-\gamma) <A<N-1$, then
\[
M(A/\left|x\right|,\gamma,\lambda) >1.
\]
\item[$\textcolor{gray}{\bullet}$] In the cases: $A=N-1$ and $\lambda\ge 1+\gamma$; $A=(\lambda-\gamma)(N-1)<N-1$, then
\[
M(A/\left|x\right|,\gamma,\lambda) =1.
\]
\item[$\textcolor{gray}{\bullet}$] In the case $A<\min\{N-1,(\lambda-\gamma)(N-1)\}$ then
\[
M(A/\left|x\right|,\gamma,\lambda) <1.
\]
\end{itemize}

We conclude that
\[M(A/\left|x\right|,\gamma,\lambda)=\max\left\{\frac{A}{N-1}, \frac1\lambda\left[\frac{A}{N-1}+\gamma\right]\right\}\]
\end{remark}

\subsection{A variational approach}
\label{sec:variational}
In the case $\lambda(x)=\lambda$ positive constant, and $g\equiv 0$, the argument used in \cite{dpnot}
allows to prove that the functional
	\[
	J_{p}(u)=\frac{1}{p}\int_{\Omega}\left|\nabla u\right|^{p} \ dx +\frac{\lambda}{p}\int_{\partial\Omega}\left|u\right|^{p} \ d\mathcal H^{N-1} -\int_{\Omega}fu \ dx
	\]
$\Gamma$-converges in $BV(\Omega)$ to
\[
J(u)=\int_\Omega|D u|+\min\{\lambda,1\}\int_{\de\Omega}|u| \ d\mathcal H^{N-1}-\int_{\Omega}fu \ dx.
\]
Let us observe that the minimizers of $J_{p}$ in $W^{1,p}(\Omega)$ are solutions to \eqref{pbmain}. Formally, \eqref{pb1} is the Euler-Lagrange equation related to $J$. Then, if $M(f,0,\lambda)\le 1$ it follows that the minimizers of $J_{p}$ converge (in $BV(\Omega)$) to a minimizer of $J$.

Actually, the truncation appearing in the boundary datum seems to be natural. If one considers $\lambda>1$ and the functional
\[
\tilde J(u)=\int_\Omega|D u|+\lambda\int_{\de\Omega}|u| \ d\mathcal H^{N-1}-\int_{\Omega}fu \ dx,
\]
it is easy to convince that
\begin{equation}
\label{JJtilde}
\min_{u\in BV(\Omega)} \tilde J(u)= \min_{u\in BV(\Omega)} J(u)\,.
\end{equation}
Indeed, if $v$ is a minimum for $J$, Theorem $3.1$ of \cite{ls} assures the existence of a sequence $v_k\in C^\infty_c(\Omega)$ which converges to $v$ in  $L^q(\Omega)$ for any $q\le \frac{N}{N-1}$ and such that $\int_\Omega|\nabla v_k|\, dx$ converges to $\int_{\R^N} |Dv|$ as $k\to \infty$. Hence $\tilde J(v_k)=J(v_k)$ for all $k>0$ and $\min_{u\in BV(\Omega)} \tilde J(u)\le \lim_{k\to +\infty} J(v_{k})= \min_{u\in BV(\Omega)} J(u)$.
Being the reverse inequality trivial, it holds \eqref{JJtilde}.

\subsection{A sharp estimate on $M(f,g,\lambda)$}
	Let $f\equiv 1$, $g\equiv 0$, $\lambda \ge 0$ and let $\Omega$ be a Lipschitz bounded domain. Then, by pointing out the dependence of $M$ by $\Omega$,
	\[
	M(1,0,\lambda)= M(1,0,\lambda,\Omega)= \sup_{u\in W^{1,1}(\Omega)\setminus\{0\}}\left\{\frac{\displaystyle\int_{\Omega}\left|u\right|dx}{\displaystyle\int_{\Omega}\left|\nabla u\right|dx+\lambda\int_{\partial\Omega}\left|u\right|d\mathcal H^{N-1}}\right\}.
	\]
	In this case we denote by $\Lambda(\Omega,\lambda)=\frac{1}{M(1,0,\lambda,\Omega)}$, and the value $\Lambda(\Omega,\lambda)$ is the limit, as $p\to 1$, of the first Robin $p$-Laplace eigenvalue (see \cite{dpnot}).
	It has been proved in \cite{dpnot} that when $\lambda>0$ and $\Omega$ is a Lipschitz bounded domain, then $\Lambda(\Omega,\lambda)\in ]0,+\infty[$ and
	\begin{equation}
		\label{dpnot}
		\Lambda(\Omega,\lambda)\ge \min\{\lambda,1\}\frac{N}{R},
	\end{equation}
	where $R$ is the radius of the ball having the same volume than $\Omega$. Moreover, for any $\lambda \ge 0$, inequality \eqref{dpnot} is an equality when $\Omega$ is a ball. Then \eqref{dpnot} gives an explicit upper bound for $M(1,0,\lambda,\Omega)$, and then an explicit condition in order to obtain that the solutions $u_{p}$ of \eqref{pbmain}, for this particular choice of the coefficients, go to zero in $\Omega$ as $p\to 1$.

\appendix
\section{Some auxiliary lemmas}


For the convenience of the reader, here we consider some technical lemmas used throughout the paper.

\begin{proposition}\label{holder}
	Let $1<p, p'<\infty$ be such that $\displaystyle \frac1{p}+\frac1{p'}=1$.
	Assume that $f_1, f_2\>:\>\Omega\to\mathbb R$ and $g_1, g_2\>:\>\partial\Omega\to\mathbb R$ are measurable functions satisfying
	$f_1\in L^p(\Omega)$, $f_2\in L^{p'}(\Omega)$, $g_1\in L^p(\partial\Omega, \lambda)$ and $g_2\in L^{p'}(\partial\Omega, \lambda)$. Then $f_1f_2\in L^1(\Omega)$, $g_1g_2\in L^1(\partial\Omega, \lambda)$ and
	\begin{multline*}
		\int_\Omega|f_1f_2|\, dx+\int_{\partial\Omega}\lambda(x)|g_1g_2|\, d\mathcal H^{N-1}\\
		\le \left[\int_\Omega|f_1|^p\, dx+\int_{\partial\Omega}\lambda(x)|g_1|^p\, d\mathcal H^{N-1}\right]^{\frac1p}\left[\int_\Omega|f_2|^{p'}\, dx+\int_{\partial\Omega}\lambda(x)|g_2|^{p'}\, d\mathcal H^{N-1}\right]^{\frac1{p'}}.
	\end{multline*}
\end{proposition}

\begin{proof}
	For every $\epsilon>0$, we apply Young's inequality to get
	\begin{multline*}
		\int_\Omega|f_1f_2|\, dx+\int_{\partial\Omega}\lambda(x)|g_1g_2|\, d\mathcal H^{N-1}
		\le \frac{\epsilon^p}{p}\int_\Omega|f_1|^p\, dx+\frac1{\epsilon^{p'}p'}\int_\Omega|f_2|^{p'}\, dx\\
		+\frac{\epsilon^p}{p}\int_{\partial\Omega}\lambda(x)|g_1|^p\, d\mathcal H^{N-1}+\frac1{\epsilon^{p'}p'}\int_{\partial\Omega}\lambda(x)|g_2|^{p'}\, d\mathcal H^{N-1}\,,
	\end{multline*}
	Denoting
	\[A^p=\int_\Omega|f_1|^p\, dx+\int_{\partial\Omega}\lambda(x)|g_1|^p\, d\mathcal H^{N-1}\]
	and
	\[B^{p'}=\int_\Omega|f_2|^{p'}\, dx+\int_{\partial\Omega}\lambda(x)|g_2|^{p'}\, d\mathcal H^{N-1}\]
	we have obtained that
	\[\int_\Omega|f_1f_2|\, dx+\int_{\partial\Omega}\lambda(x)|g_1g_2|\, d\mathcal H^{N-1}
	\le\frac{\epsilon^p}{p}A^p+\frac1{\epsilon^{p'}p'}B^{p'}\]
	for all $\epsilon>0$. Minimizing in $\epsilon$, it follows that
	\[\int_\Omega|f_1f_2|\, dx+\int_{\partial\Omega}\lambda(x)|g_1g_2|\, d\mathcal H^{N-1}
	\le AB\]
	as desired.
\end{proof}

\begin{proposition}\label{bounded}
	Assume that $f\in L^s(\Omega)$ and $g\in L^s(\partial\Omega, \lambda)$ for all $1\le s<\infty$. If
	\[\int_\Omega|f|^s\, dx+\int_{\partial\Omega}\lambda(x) |g|^s\, d\mathcal H^{N-1}\le C^s\qquad\forall s<\infty\]
	for some constant $C>0$,
	then
	\begin{enumerate}
		\item There exists $\displaystyle \lim_{s\to\infty}\left[\int_\Omega|f|^s\, dx+\int_{\partial\Omega}\lambda(x) |g|^s\, d\mathcal H^{N-1}\right]^{\frac1s}$;
		\item $f\in L^\infty(\Omega)$ and $g\chi_{\{\lambda>0\}}\in L^\infty(\partial\Omega)$;
		\item $\displaystyle \max\{\|f\|_\infty, \|g\chi_{\{\lambda>0\}}\|_\infty\}=\lim_{s\to\infty}\left[\int_\Omega|f|^s\, dx+\int_{\partial\Omega}\lambda(x) |g|^s\, d\mathcal H^{N-1}\right]^{\frac1s}$.
	\end{enumerate}
\end{proposition}

\begin{proof}
	(1) Let $\Lambda=|\Omega|+\int_{\partial\Omega}\lambda(x)\, d\mathcal H^{N-1}$. Observe that the family
	\[\left[\frac1\Lambda\left(\int_\Omega|f|^s\, dx+\int_{\partial\Omega}\lambda(x) |g|^s\, d\mathcal H^{N-1}\right)\right]^{\frac1s}\]
	is increasing in $s$ as a consequence of Proposition \ref{holder}. On the other hand, it is bounded by
	$\displaystyle \frac{C}{\Lambda^{1/s}}\le C+1$
	for $s$ large enough. Hence, it converges. Denote
	\[\Gamma =\lim_{s\to\infty}\left[\frac1\Lambda\left(\int_\Omega|f|^s\, dx+\int_{\partial\Omega}\lambda(x) |g|^s\, d\mathcal H^{N-1}\right)\right]^{\frac1s}\]
	and notice that it leads to
	\[\lim_{s\to\infty}\left[\int_\Omega|f|^s\, dx+\int_{\partial\Omega}\lambda(x) |g|^s\, d\mathcal H^{N-1}\right]^{\frac1s}=\Gamma.\]
	
	\bigskip
	
	(2) We are proving that $|f(x)|\le \Gamma$ a.e. in $\Omega$. For every $\epsilon>0$, define
	\[A_\epsilon=\{x\in\Omega\>:\> |f(x)|>\Gamma+\epsilon\}.\]
	If $|A_\epsilon|>0$, then
	\[\left[\int_\Omega|f|^s\, dx+\int_{\partial\Omega}\lambda(x) |g|^s\, d\mathcal H^{N-1}\right]^{\frac1s}\ge \left[\int_{A_\epsilon}|f|^s\, dx\right]^{\frac1s}\ge (\Gamma+\epsilon)|A_\epsilon|^{\frac1s}.\]
	Letting $s$ go to $\infty$, we arrive at $\Gamma\ge\Gamma+\epsilon$, which is a contradiction. So $A_\epsilon$ is a null set and consequently
	$|f(x)|\le \Gamma+\epsilon$ a.e. for every $\epsilon>0$, wherewith $|f(x)|\le \Gamma$ a.e.
	
	We next check that $|g(x)|\le \Gamma$ $\mathcal H^{N-1}$--a.e. on $\{\lambda>0\}$ following a similar argument. For every $\epsilon>0$, define
	\[B_\epsilon=\{x\in\partial\Omega\>:\>\lambda(x)>0\,, \quad |g(x)|>\Gamma+\epsilon\}.\]
	If $\int_{B_\epsilon}\lambda\, d\mathcal H^{N-1}>0$, then
	\[\left[\int_\Omega|f|^s\, dx+\int_{\partial\Omega}\lambda(x) |g|^s\, d\mathcal H^{N-1}\right]^{\frac1s}\ge \left[\int_{B_\epsilon}\lambda |g|^s\, d\mathcal H^{N-1}\right]^{\frac1s}\ge (\Gamma+\epsilon)\int_{B_\epsilon}\lambda \, d\mathcal H^{N-1}.\]
	When $s$ goes to $\infty$, this inequality becomes $\Gamma\ge\Gamma+\epsilon$, which is a contradiction. So $\int_{B_\epsilon}\lambda\, d\mathcal H^{N-1}$ vanishes. Hence  $|g(x)|\chi_{\{\lambda>0\}}\le \Gamma+\epsilon$ for all $\epsilon>0$, so that $|g(x)|\chi_{\{\lambda>0\}}\le \Gamma$.

	\bigskip
	
	(3) By the previous point, we already know that $\max\{\|f\|_\infty, \|g\chi_{\{\lambda>0\}}\|_\infty\}\le \Gamma$. The reverse inequality follows from the inequality
	\[\left[\int_\Omega|f|^s\, dx+\int_{\partial\Omega}\lambda(x) |g|^s\, d\mathcal H^{N-1}\right]^{\frac1s}\le \max\{\|f\|_\infty, \|g\chi_{\{\lambda>0\}}\|_\infty\}\Lambda^{\frac1s}\]
	by taking the limit as $s$ tends to $\infty$.
\end{proof}

\section*{Funding}

F. Della Pietra has been partially supported by the  MIUR-PRIN 2017 grant ``Qualitative and quantitative aspects of nonlinear PDE's'', by GNAMPA of INdAM, by  the FRA Project (Compagnia di San Paolo and Universit\`a degli studi di Napoli Federico II) \verb|000022--ALTRI_CDA_75_2021_FRA_PASSARELLI|.

F. Oliva has been partially supported by GNAMPA of INdAM and by PON Ricerca e Innovazione 2014-2020.

S. Segura de Le\'on has been supported  by MCIyU \& FEDER, under project PGC2018--094775--B--I00 and by CECE (Generalitat Valenciana) under project AICO/2021/223.

\end{document}